\documentclass{tran-l}

\usepackage{amssymb}
\usepackage{amsmath}

\usepackage{diagrams}

\newtheorem{theorem}{Theorem}[section]

\newtheorem{corollary}[theorem]{Corollary}

\newtheorem{definition}[theorem]{Definition}

\newtheorem{lemma}[theorem]{Lemma}

\newtheorem{proposition}[theorem]{Proposition}
\newtheorem{remark}[theorem]{Remark}

\newarrow {Partial} ----{harpoon}
\newarrow{Into} C--->

\begin{document}

\title{Principal Bundles as Frobenius adjunctions with application to geometric morphisms}

\author{Christopher F. Townsend}
\address{8 Aylesbury Road, Tring, Hertfordshire, HP23 4DJ, U.K.}

\maketitle

\begin{abstract}
Using a suitable notion of principal $G$-bundle, defined relative to an arbitrary cartesian category, it is shown that principal bundles can be characterised as adjunctions that stably satisfy Frobenius reciprocity. The result extends from $G$, an internal group, to $\mathbb{G}$ an internal groupoid. Since geometric morphisms can be described as certain adjunctions that are stably Frobenius, as an application it is proved that all geometric morphisms, from a localic topos to a bounded topos, can be characterised as principal bundles.
\end{abstract}

\section{Introduction}

The main purpose of this paper is to show that in any cartesian category $\mathcal{C}$, principal $G$-bundles over an object $X$ for an internal group $G$ are the same thing as adjunctions $\mathcal{C}/X \pile{\rTo \\ \lTo} [ G, \mathcal{C}]$ over $\mathcal{C}$ that stably satisfy Frobenius reciprocity, provided the adjunction of connected components, $ \Sigma_G \dashv G^*:[ G, \mathcal{C}] \pile{\rTo \\ \lTo} \mathcal{C}$, exists and itself stably satisfies Frobenius reciprocity. $[ G, \mathcal{C}] $ is the category of objects of $\mathcal{C}$ equipped with a $G$ action; i.e. the category of $G$-objects with $G$-homomorphisms between them.

Geometric morphisms can be characterised as adjunctions between categories of locales that satisfy Frobenius reciprocity, \cite{towgeom}. So as an application to the case $\mathcal{C}=\mathbf{Loc}$, it follows that geometric morphisms $Sh(X) \rTo B(G) $, from the category of sheaves over a locale $X$ to the topos a $G$-sets, for any localic group $G$, are the same thing as localic principal $\hat{G}$-bundles, where $\hat{G}$ is the \'{e}tale completion of $G$. This is a key relationship as it can be used to establish, for discrete $G$ at least, the more well known result that there is a classifying space for principal $G$-bundles; see \cite{IM_classtop} for a description of how topos theoretic results about principal bundles relate back to more well known topological results.

Our main result easily generalises from internal groups to internal groupoids. It follows that any geometric morphism from a localic topos to a topos bounded over some base topos $\mathbf{Set}$ can be represented as a principal bundle.

In the next section we recall some basic facts about the category $[ G, \mathcal{C}]$ of $G$-objects and $G$-homomorphisms for a group $G$ internal to a cartesian category $\mathcal{C}$ and define a notion of principal $G$-bundle over an object $X$ of $\mathcal{C}$.

In the third section we prove our main result which shows how the notion of principal $G$-bundle can be related to stably Frobenius adjunctions. The proofs and techniques are simple as they only involved cartesian categories and various adjunctions. Our strategy is to first demonstrate the main result for the case of principal bundles over the terminal object $1$ (i.e. $X=1$) and then show how the case of general $X$ can be obtained by applying the proof for $X=1$ to the cartesian category $\mathcal{C}/X$.

The fourth section describes in summary how the main result generalises to groupoids.

The fifth section describes how the main result can be applied to the case $\mathcal{C} = \mathbf{Loc}$, the category of locales, to give a description of geometric morphisms $Sh(X) \rTo B{\mathbb{G}} $ for certain classes of localic groupoids $\mathbb{G}$.

The results apply equally to open localic groupoids and to proper localic groupoids. In fact, an axiomatic treatment of locale theory (\cite{tow_weakt}) reveals that the theory of `open' principal bundles can be viewed as order dual to the theory of `proper' principal bundles. The results here show that both theories of principal bundles have representations as Frobenius adjunctions. What is not clear is whether the theory of `proper' principal bundles has anything like the depth of the more familiar theory of `open' principal bundles.

\section{Principal G-bundles in a cartesian category}

We start with some basic definitions and results relative to a cartesian category, $\mathcal{C}$. If $(G,m)$ is an internal group then
$[G, \mathcal{C}]$ is the category of {\em $G$-objects}, whose objects are pairs $(A,*_A)$ where $A$ is an object of $\mathcal{C}$ and $*_A: G \times A \rTo A$ is a $G$-action; that is, satisfies the usual unit and associative diagrams. For example, $(G,m)$ itself is a $G$-object; further for any object $X$ of $\mathcal{C}$, $(X, \pi_2)$ is an object of $[ G, \mathcal{C}]$; it is $X$ with the {\em trivial} action. The morphisms $f: (A, *_A) \rTo (B, *_B)$ of $[G , \mathcal{C}]$ are morphisms $f: A \rTo B$ that commute with the actions, i.e. $f *_A   = *_B ( Id_G \times f)$. Sending any $X$ to $(X, \pi_2)$ defines a functor $G^*$ from $\mathcal{C}$ to $[G , \mathcal{C}]$. Its left adjoint, when it exists, is written $\Sigma_G$ and must send $(A,*_A)$ to the coequalizer of $\pi_2,*_A: G \times A \pile{\rTo \\ \rTo} A$. If $\Sigma_G$ exists then $\Sigma_G(G,m)=1$ because $!:G \rTo 1$ is a coequalizer of $\pi_2,m:G \times G \pile{ \rTo \\ \rTo} G$ (it is split by the identity $e:1 \rTo G$ of $G$).

$[G , \mathcal{C}]$ is a cartesian category; products and equalizers are created in $\mathcal{C}$. $(G,m)$ is a rather special object of $[G , \mathcal{C}]$; for any other object $(A, *_A)$, $(A, *_A) \times (G,m) \cong (A, \pi_2) \times (G,m)$. To see this send an `element' $(a,g)$ of $(A, \pi_2) \times (G,m)$ to $(g*_Aa,g)$ and an `element' $(a,g)$ of $(A, *_A) \times (G,m)$ to $(g^{-1} *_A a , g)$; it is easy to verify that this establishes an isomorphism in $[G , \mathcal{C}]$. Although this argument, and arguments below, deploy `elements' it is important to understand that this is just shorthand for defining and arguing about morphisms in a category.

If $X$ is an object of $\mathcal{C}$ then the {\em slice category}, written $\mathcal{C}/X$, is the category whose objects are morphisms $f:Y \rTo X$ and whose morphisms are commuting triangles. We will tend to use the notation $Y_f$ when considering the morphism $f: Y \rTo X$ as an object of $\mathcal{C}$. Any morphism $f: Y \rTo X$ of $\mathcal{C}$ gives rise to an adjunction $\Sigma_f \dashv f^* : \mathcal{C}/Y \pile{ \rTo \\ \lTo } \mathcal{C}/X$ between slice categories where the right adjoint is given by pullback (and $\Sigma_f(Z_g) = Z_{fg}$ for a morphism $g:Z \rTo Y$). $\mathcal{C}/X$ is a cartesian category; limits are created in $\mathcal{C}$. Coequalizers in $\mathcal{C}/X$, when they exists, are created in $\mathcal{C}$. If $G=(G,m,e)$ is an internal group of $\mathcal{C}$ and $X$ is an object of $\mathcal{C}$ then $G \times X$ is an internal group of $\mathcal{C}/X$; its multiplication is given by $(G \times G) \times X \rTo{m \times Id_X} G \times X$ and its unit is $X \rTo^{(e!^X,Id_X)} G \times X $.

A morphism $f:X \rTo Y$ of $\mathcal{C}$ is an {\em effective descent morphism} if the pullback functor $f^* : \mathcal{C}/Y  \rTo  \mathcal{C}/X$ is monadic. Since $f^*$ always has a left adjoint, by Beck's monadicity theorem, $f$ is an effective descent morphism if and only if $f^*$ reflects isomorphisms and $\mathcal{C}/Y$ has and $f^*$ preserves coequalizers for any pair of $f^*$-split arrows. For any internal group $G$ in a cartesian category the morphism $!:(G,m) \rTo 1$ of $[G, \mathcal{C}]$ is an effective descent morphism. This can be observed because of the well known fact that $[G , \mathcal{C}]/(G,m) \simeq \mathcal{C}$ (to see this send a morphism to its kernel in one direction and send an object $X$ of $\mathcal{C}$ to the projection $(X,\pi_2) \times (G,m) \rTo (G,m)$ in the other). Under this equivalence the pullback functor $(G,m)^*:[G,\mathcal{C}] \rTo $$ [G,\mathcal{C}] /(G,m) $ is just the forgetful functor from $[G,\mathcal{C}]$ to $\mathcal{C}$ that forgets the group action; its left adjoint sends $X$ to $(G,m) \times (X, \pi_2) $ and this adjunction induces a monad on $\mathcal{C}$; it is easy to see that $[G,\mathcal{C}]$ is by definition the category of algebras of this induced monad.

An adjunction $L\dashv R: \mathcal{D} \pile{\rTo \\ \lTo} \mathcal{C}$ between cartesian categories satisfies \emph{Frobenius reciprocity} provided the morphism $L(R(X)\times W)\rTo^{(L\pi _{1},L\pi _{2})}LRX\times LW\rTo^{\varepsilon _{X}\times Id_{LW}}X\times LW$ is an isomorphism for all objects $W$ and $X$ of $\mathcal{D}$ and $\mathcal{C}$ respectively where $\varepsilon$ is the counit of the adjunction. For any object $X$ of $\mathcal{C}$ there is an adjunction $L_X\dashv R_X: \mathcal{D}/RX \pile{\rTo \\ \lTo} \mathcal{C}/X$ given by $L_X(W_g)=$ `the adjoint transpose of $g$' and $R_X(Y_f)=R(f)$. The original adjunction $L\dashv R$ is said to be \emph{stably Frobenius} provided $L_X\dashv R_X $ satisfies Frobenius reciprocity for every object $X$ of $\mathcal{C}$. It is easy to verify that for any morphism $f: X \rTo Y$ of a cartesian category the pullback adjunction $\Sigma_f \dashv f^* : \mathcal{C}/X \pile{ \rTo \\ \lTo } \mathcal{C}/Y$ is stably Frobenius. Notice that both the property of satisfying Frobenius reciprocity and of being stably Frobenius are stable under composition of adjunctions. Given two adjunctions $\mathcal{D} \pile{\rTo^L \\ \lTo_R} \mathcal{C}$ and $\mathcal{D}' \pile{\rTo^{L'} \\ \lTo_{R'}} \mathcal{C}$ then any third adjunction $F \dashv U : \mathcal{D} \pile{\rTo \\ \lTo} \mathcal{D}'$ is said to be {\em over $\mathcal{C}$} provided $L'F=L$; of course, in such circumstances $UR' \cong R$ by uniqueness of adjoints. The collection all adjunctions between $\mathcal{D}$ and $\mathcal{D}'$ over $\mathcal{C}$ can be considered as a category with morphisms natural transformations between the left adjoints.

Our first lemma shows that in certain situations adjunctions that satisfy Frobenius reciprocity and are over a base category $\mathcal{C}$ give rise to effective descent morphisms:
\begin{lemma}\label{descent_lemma}
$G$ is an internal group in a cartesian category $\mathcal{C}$ such that $G^*: \mathcal{C} \rTo $ $ [ G ,\mathcal{C}] $ has a left adjoint $\Sigma_G$ and the resulting adjunction satisfies Frobenius reciprocity. $L \dashv R: \mathcal{C} \pile{ \rTo \\ \lTo} [G, \mathcal{C}]$ is an adjunction over $\mathcal{C}$ (i.e. $\Sigma_G L = Id_{\mathcal{C}}$) which also satisfies Frobenius reciprocity. Write $(P,*)$ for the $G$-object $L1$ and assume further that $P \cong R(G,m)$. Then $!: P \rTo 1$ is an effective descent morphism.
\end{lemma}
We will see in the next section that, in fact, the condition $P \cong R(G,m)$ always holds.
\begin{proof}
Firstly $\Sigma_GL1 = 1$ by assumption that $L \dashv R$ is over $\mathcal{C}$. So for any object $X$ of $\mathcal{C}$, $\Sigma_G(L1 \times G^* X) \cong \Sigma_GL1 \times X \cong X$; i.e.
\begin{eqnarray*}
G \times P \times X \pile{ \rTo^{* \times Id_X} \\ \rTo_{\pi_{23}}} P \times X \rTo^{\pi_2} X
\end{eqnarray*}
is a coequalizer diagram in $\mathcal{C}$. Since this is a coequalizer for every $X$ it is easy to see that $P^* : \mathcal{C} \rTo \mathcal{C}/P$ reflects isomorphisms.
So to complete the proof all we need to show is that if $X \pile{\rTo^f \\ \rTo_g } Y $ is pair of morphisms of $\mathcal{C}$ with the property that there is a split coequalizer diagram
\begin{eqnarray*}
P \times X \pile{\rTo^{Id \times f} \\ \rTo_{Id \times g} \\ \lTo_s} P \times Y \pile{ \rTo^q \\ \lTo_i} Q  \text{ \em \em \em    (*)}
\end{eqnarray*}
 in $\mathcal{C}/P$ then there is a coequalizer $Y \rTo^n N$ of $f$ and $g$ in $\mathcal{C}$ with the property that $ P \times Y \rTo^{Id_P \times n} P \times N$ is isomorphic to $P \times Y \rTo^q Q$.

 Since $P \cong R(G,m)$ by taking the adjoint transpose of (*) (across $L \dashv R$) and applying the Frobenius reciprocity assumption we obtain a split coequalizer diagram:
\begin{eqnarray*}
(G,m)  \times LX \pile{\rTo^{Id \times Lf} \\ \rTo_{Id \times Lg} \\ \lTo_{s'}} (G,m) \times LY \pile{ \rTo^{q \circ \cong} \\ \lTo_{i'}} LQ
\end{eqnarray*}
Since $(G,m) \rTo 1$ is an effective descent morphism, there is a coequalizer diagram
\begin{eqnarray*}
LX \pile{ \rTo^{Lf} \\ \rTo_{Lg}} LY \rTo^t (T,*_T)
\end{eqnarray*}
in $[ G ,\mathcal{C}]$ with the property that $(G,m) \times LY \rTo^{Id \times t} (G,m) \times (T,*_T) $ is isomorphic to $(G,m) \times LY \rTo^{Lq \circ \cong} LQ$. Because $\Sigma_GL  = Id_{\mathcal{C}}$, $\Sigma_G t $ is a morphism from Y to $\Sigma_G(T, *_T)$. Now for any morphism $Y \rTo^h Z$ of $\mathcal{C}$, because $Z \cong RG^*Z$ we have that $h$ composes equally with $f$ and $g$ (i.e. $hf=hg$) if and only if $\hat{h} Lf = \hat{h} Lg$ (where $\hat{h}$ is the adjoint transpose of $Y \rTo^h Z \rTo^{\cong} RG^*Z)$ if and only if $\hat{h}:LY \rTo G^* Z$ factors through $(T,*_T)$, if and only if $Y \rTo^h Z$ factors through $\Sigma_G (T,*_T)$, where for the last implication we are taking adjoint transpose under $\Sigma_G\dashv G^*$. It follows that $Y \rTo^{\Sigma_G(t)} \Sigma_G(T,*_T)$ is a coequalizer in $\mathcal{C}$ of $f$,$g$. Notice that $(T, *_T) \cong L \Sigma_G ( T , *_T)$ because $L$, as a left adjoint, preserves coequalizers. Finally, for any object $W$ of $\mathcal{C}$, morphisms $Q \rTo W$ correspond to morphisms $P \times Y \rTo W$ that compose equally with $Id \times f$ and $Id \times g$ and these in turn correspond (under $L \dashv R$, using $W \cong RG^*W$) to morphisms $(G,m) \times LY \rTo G^*W$ that compose equally with $Id \times Lf$ and $Id \times Lg$. These then correspond to morphisms $(G,m) \times (T , *_T) \rTo G^*W$ since $LQ \cong (G,m) \times (T,*_T)$. Then, by adjoint transpose under $\Sigma_G \dashv G^*$, these correspond to morphisms $ \Sigma_G ( (G,m) \times (T, *_T)) \rTo W$. But
\begin{eqnarray*}
\Sigma_G((G,m) \times (T, *_T)) & \cong & \Sigma_G (( G,m) \times L\Sigma_G (T, *_T)) \\
                                & \cong & \Sigma_G ( ( G,m) \times (P,*) \times G^* \Sigma_G (T , *_T)) \\
                                & \cong & \Sigma_G (( G,m) \times (P, \pi_2) \times G^* \Sigma_G (T *_T)) \\
                                & \cong & \Sigma_G ( ( G,m) \times G^* ( P \times \Sigma_G(T,*_T))) \\
                                & \cong & \Sigma_G(G,m) \times P \times \Sigma_G(T, *_T) \\
                                & \cong & P \times \Sigma_G(T,*_T)\\
\end{eqnarray*}
and so $Q \cong P \times \Sigma_G ( T , *_T)$ as required.
\end{proof}
We now define principal bundle relative to an arbitrary cartesian category. The definition at this level of generality appears to be originally in \cite{KockGen}.
\begin{definition}
If $G$ is an internal group in a cartesian category $\mathcal{C}$ then a {\em principal $G$-bundle} is a $G$-object $(P,*)$ such that

(i) $!:P \rTo 1$ is an effective descent morphism; and,

(ii) the morphism $(*,\pi_2):G \times P \rTo P \times P$ of $\mathcal{C}$ is an isomorphism.

\end{definition}
The inverse of $(*,\pi_2)$, if it exists, must be a map of the form $(\psi, \pi_2)$ for a morphism $\psi:P \times P \rTo G$. For any `elements' $b$ and $b'$ of $B$, $\psi(b,b')$ is the unique `element' of $G$ such that $\psi(b,b') * b' = b$. $\psi$ has a number of well known properties that will be exploited below; for example, $\psi(g*p,p')=g*\psi(p,p')$ and $\psi(p,g*p')=\psi(p,p')g^{-1}$.

The \emph{category of principal $G$-bundles} is the full subcategory of $[G, \mathcal{C}]$ consisting of objects that are principal $G$-bundles.
\begin{definition}
If $G$ is an internal group in a cartesian category $\mathcal{C}$ and $X$ is an object of $\mathcal{C}$ then a {\em principal $G$-bundle over $X$} is a $G$-object $(P,*)$, together with a morphism $f: P \rTo X$ such that

(i) $f*=f\pi_2$; i.e. $f(g*p)=f(p)$ for any `elements' $g$, $p$ of $G$, $P$ respectively,

(ii) $f:P \rTo X$ is an effective descent morphism; and,

(iii) the morphism $(*,\pi_2):G \times P \rTo P \times_X P$ of $\mathcal{C}/X$ is an isomorphism.

\end{definition}
In our general context of cartesian categories there is no real extra generality when talking about principal bundles over $X$ in comparison to principal bundles:
\begin{lemma}
If $\mathcal{C}$ is a cartesian category, $G$ an internal group and $X$ an object of $\mathcal{C}$, then (i) $[G \times X , \mathcal{C}/X] \cong [G,\mathcal{C}]/(X , \pi_2)$ and (ii) the category of principal $G$-bundles over $X$ is isomorphic to the category of $G \times X$ principal bundles relative to $\mathcal{C}/X$.
\end{lemma}
\begin{proof}
(i) can be checked from the definitions and (ii) follows from (i).
\end{proof}
We will use this lemma to ease the proof of our main theorem, which is the purpose of the next Section.

\section{A categorical relationship between principal bundles and Frobenius reciprocity}

We can now state and prove our main result for the case $X=1$; this will be used in the proof for general $X$ to follow.

\begin{proposition}\label{X=1}
$\mathcal{C}$ is a cartesian category and $G$ is an internal group with the property that the functor $G^* : \mathcal{C} \rTo $$  [ G, \mathcal{C}]$ has a left adjoint $\Sigma_G$ such that $\Sigma_G \dashv G^* $ satisfies Frobenius reciprocity. Then there is an equivalence  between the category of  principal $G$-bundles and the category of adjunctions $L \dashv R : \mathcal{C} \pile { \rTo \\ \lTo} [G, \mathcal{C}]$ over $\mathcal{C}$ that satisfy Frobenius reciprocity.

Further any such adjunction is also stably Frobenius.
\end{proposition}

\begin{proof}
Say $L \dashv R : \mathcal{C} \pile { \rTo \\ \lTo} [G, \mathcal{C}]$ satisfies Frobenius reciprocity and has $\Sigma_G L = Id_{\mathcal{C}}$. Let $L1 = (P, *)$. Then $LR(G,m) \cong (P, *) \times (G,m)$, which we have observed already is isomorphic to $G^*P \times (G,m)$. By assumption that $\Sigma_G L = Id_{\mathcal{C}}$ we have that for any object $X$ of $\mathcal{C}$, $X \cong R G^* X$ and so further $LR(G,m) \cong L(1 \times  R G^* R(G,m) \cong (P , * ) \times G^* R(G,m)$. But $\Sigma_G(G,m) \cong 1$ and $\Sigma_G(P,*) = 1$, the latter because $\Sigma_GL1 = 1$. It follows that $R(G, m) \cong P$ because $\Sigma_G$ satisfies Frobenius reciprocity, and this exhibits an isomorphism $G \times P \cong P \times P$. By Lemma \ref{descent_lemma}, $!^P : P \rTo 1 $ is an effective descent morphism; therefore $(P, *)$ is a principal bundle.

In the other direction, say we are given a principal bundle $(P, *)$. We will use $\psi: P \times P \rTo G$ for the map that exists because $G \times P \cong P \times P$. Define $L: \mathcal{C} {\rTo} [ {G} , \mathcal{C}]$ by $LX = (P , *) \times (X , \pi_2)$. Define $R: [ G , \mathcal{C} ]  \rTo \mathcal{C}$ by sending $(A, *_A)$ to the coequalizer of $P \times A$ defined by the arrows
\begin{eqnarray*}
G \times P \times A \pile { \rTo^{ * \times Id_A } \\ \rTo_{(Id_P \times *_A)(Id_P \times i \times Id_A)(\tau \times Id_A)}} P \times A
\end{eqnarray*}
where $\tau: G \times A \rTo A \times G$ is the twist isomorphism and $i:G \rTo G$ is the inverse of $G$. In other words $R(A, *_A)$ is defined to be the tensor $P \otimes_G A$ where $(g * p) \otimes a = p \otimes (g^{-1}*_A a)$ for any `elements' $a$, $p$ and $g$ of $A$, $P$ and $G$ respectively. This coequalizer exists because an easy diagram chase shows that it is isomorphic to $\Sigma_G ( (P,*) \times (A , *_A ))$. There is an `evaluation' map $ev: P \times (P \otimes_G A) \rTo A$ defined by $(b',b \otimes a) \mapsto \psi(b',b)*_A a$. This is well defined because the coequalizer that defines $P \otimes_G A$ is stable under products; this is because $\Sigma_G \dashv G^*$ satisfies Frobenius reciprocity. Using properties of $\psi$ it can be checked that $ev: (P,*) \times (P \otimes_{G} A , \pi_2) \rTo (A, *_A)$; i.e. the evaluation map is a $G$-homomorphism. We now check that $L$ is left adjoint to $R$. Say we are given an object $X$ of $\mathcal{C}$ and an object $(A,*_A)$ of $[G, \mathcal{C}]$, then send any map $f:X \rTo P \otimes_{G} A$ to the morphism
\begin{eqnarray*}
P \times X \rTo^{Id_P \times f} P \times (P \otimes_{G} A) \rTo^{ev} A \text{.}
\end{eqnarray*}
It is easy to check this is a $G$-homomorphism (from  $(P,*) \times (X , \pi_2)$ to $(A,*_A)$) because $ev(g*p,p^x \otimes a^x)=\psi(g*p,p^x)*_A a^x = (g\psi(p,p^x)) *_A a^x = g *_A (\psi(p,p^x) *_A a^x ) = g *_A ev (p,p^x \otimes a^x)$ where $f(x) = p^x \otimes a^x$. On the other hand given any $G$-homomorphism $g:(P,*) \times (X, \pi_2) \rTo (A,*_A)$ notice that the map
\begin{eqnarray*}
P \times X \rTo^{(\pi_1,g)} P \times A \rTo^{\otimes} P \otimes_{G} A
\end{eqnarray*}
composes equally with $* \times Id_X :G \times P \times X \rTo P \times X$ and  $\pi_2 \times Id_X :G \times P \times X \rTo P \times X$ and so factors through $\pi_2: P \times X \rTo X$ (because $\Sigma_{G}((P,*) \times (X, \pi_2)) \cong \Sigma_{G}(P,*) \times X \cong  1 \times X $). This defines a map $X \rTo  P \otimes_{G} A$. To check that this establishes a natural bijection between $\mathcal{C}(X,P \otimes_{G}A)$ and $[G,\mathcal{C}]((P,*) \times (X, \pi_2),(A,*_A))$ is a routine application of the properties of $\psi: P \times P \rTo G$. Therefore $L \dashv R$. Observe that the conunit of the adjunction is given by the evaluation map $ev: (P,*) \times (P \otimes_{G} A , \pi_2) \rTo (A, *_A)$.

We must show that $L \dashv R$ satisfies Frobenius reciprocity; i.e., that the map $(P,*) \times (X \times P \otimes_{G} A, \pi_2) \rTo  (P,*) \times (X , \pi_2) \times (A, *_A)$ given by $(p,x,p'\otimes a) \mapsto (p,x,\psi(p,p')*_A a)$ has an inverse. It is easy to check using the properties of $\psi$ that the assignment $(p,x,a) \mapsto (p,x,p \otimes a)$ defines a $G$-homomorphism and is the required inverse.

Also observe that $\Sigma_G (P,*) =1$ because $!:P \rTo 1$ is a regular epimorphism. Therefore $\Sigma_G LX = \Sigma_G ((P,*) \times (X, \pi_2)) \cong X$ and so $L \dashv R$ is over $\mathcal{C}$ as required.

It is clear that we have now established a categorical equivalence between principal $G$-bundles and adjunctions. This is because any $L \dashv R$ that satisfies Frobenius reciprocity is uniquely determined by $L1$ and, in the other direction, the principal bundle associated with the adjunction $(P,*) \times (\_,\pi_2) \dashv P \otimes_{G} (\_)$ is $(P,*)$.

Finally we prove that, in fact, the adjunction $L \dashv R$ is stably Frobenius. Let $(B, *_B)$ be an object of $[ G , \mathcal{C} ]$. We must check, for any $G$-homomorphism $n: (A, *_A) \rTo (B, *_B)$ and any $f: X \rTo P \otimes_{G} B$ that the canonical map $(P,*) \times (X \times_{P \otimes_{G} B} P \otimes_{G} A, \pi_2) \rTo ((P,*) \times (X,\pi_2)) \times_{(B, *_B)} (A , *_A)$ is an isomorphism. Given that we have already established an isomorphism $(P,*) \times (X \times P \otimes_{G} A, \pi_2) \cong  (P,*) \times (X , \pi_2) \times (A, *_A)$ this is just a question of verifying that the subobject of $(P,*) \times (X \times P \otimes_{G} A, \pi_2)$ determined by $\{ (p,x,p' \otimes a) | p^x \otimes b^x = p' \otimes n(a) \}$ corresponds under this isomorphism to the subobject $\{ (p,x,a) | \psi(p,p^x)*_B b^x = n(a) \}$ of $(P,*) \times (X , \pi_2) \times (A, *_A)$ (where we are using $p^x\otimes b^x$ for $f(x)$). It must also be verified that the isomorphism is over $(B,*_B)$. Both easily follow again from the properties of $\psi$.

\end{proof}

In the proof above we did not use the fact that $!: P \rTo 1$ is an effective descent morphism in the construction of a Frobenius adjunction from the principal bundle $(P,*)$; we only exploited the fact that it is a regular epimorphism. It follows that as a side result we immediately have the following lemma:
\begin{lemma}
$G$ is an internal group in a cartesian category $\mathcal{C}$, $(P,*)$ is a $G$-object such that the morphism $(*,\pi_2):G \times P \rTo P \times P$ of $\mathcal{C}$ is an isomorphism and $!:P \rTo 1$ is a regular epimorphism. Then, $!:P \rTo 1 $ is an effective descent morphism and $(P,*)$ is a principal $G$-bundle (provided $G$ is such that $G^*$ has a left adjoint and the resulting adjunction satisfies Frobenius reciprocity).
\end{lemma}

Our main result is now an easy application of the case $X = 1$:

\begin{theorem}
$\mathcal{C}$ is a cartesian category and $G$ is an internal group with the property that the functor $G^* : \mathcal{C} \rightarrow [ G, \mathcal{C}]$ has a left adjoint $\Sigma_G$ such that $G^* \dashv \Sigma_G$ is stably Frobenius. $X$ is an object of $\mathcal{C}$. Then there is an equivalence between the category of principal $G$-bundles over $X$ and the category of adjunctions $L \dashv R : \mathcal{C}/X \pile { \rTo \\ \lTo} [G, \mathcal{C}]$ that are stably Frobenius and are over $\mathcal{C}$ (i.e. $\Sigma_G L = \Sigma_X$).
\end{theorem}
\begin{proof}
By the proposition all that is required is a proof that the category of adjunctions $L' \dashv R' : \mathcal{C}/X \pile{\rTo \\ \lTo} [G \times X , \mathcal{C}/X]$ over $\mathcal{C}/X$ that satisfy Frobenius reciprocity  is equivalent to the category of adjunctions $L \dashv R : \mathcal{C}/X \pile { \rTo \\ \lTo} [G, \mathcal{C}]$ over $\mathcal{C}$ that are stably Frobenius. To see that this is sufficient to complete the proof recall from above that $[G, \mathcal{C}]/(X, \pi_2) \cong [ G \times X,  \mathcal{C}/X ]$ and so the assumption that $G^* \dashv \Sigma_G$ is stably Frobenius implies that $(G \times X)^*:\mathcal{C}/X \rightarrow [G \times X,\mathcal{C}/X]$ has a left adjoint and the resulting adjunction satisfies Frobenius reciprocity, allowing the proposition to be applied.
Now any adjunction $L \dashv R : \mathcal{C} /X \pile { \rTo \\ \lTo} [G, \mathcal{C}]$ over $\mathcal{C}$ factors as
\begin{eqnarray*}
\mathcal{C} /X \pile { \rTo^{\Sigma_{\Delta_X}} \\ \lTo_{\Delta^*_X} } \mathcal{C} /X \times X \pile { \rTo^{L_{(X,\pi_2)}} \\ \lTo_{R_{(X,\pi_2)}}} [G, \mathcal{C}]/(X, \pi_2)  \pile { \rTo^{\Sigma_{(X, \pi_2)}} \\ \lTo_{(X, \pi_2)^*}}[G, \mathcal{C}]
\end{eqnarray*}
and so gives rise to an adjunction $L_{(X,\pi_2)} \Sigma_{\Delta_X}  \dashv \Delta^*_X R_{(X,\pi_2)}$ which can be seen to be over $\mathcal{C}/X$; this adjunction satisfies Frobenius reciprocity because $L \dashv R$ is stably Frobenius (and the property of satisfying Frobenius reciprocity is preserved by composition of adjunctions). In the other direction say we are given $L' \dashv R' : \mathcal{C}/X \pile{\rTo \\ \lTo} [G \times X , \mathcal{C}/X]$ over $\mathcal{C}/X$ that satisfies Frobenius reciprocity. Then by the proposition $L' \dashv R'$ is stably Frobenius and so the composite adjunciton
\begin{eqnarray*}
\mathcal{C} /X \pile { \rTo^{L'} \\ \lTo_{R'} }  [G, \mathcal{C}]/(X, \pi_2)  \pile { \rTo^{\Sigma_{(X, \pi_2)}} \\ \lTo_{(X, \pi_2)^*}}[G, \mathcal{C}]
\end{eqnarray*}
is stably Frobenius. It can be readily checked that this composite adjunction is over $\mathcal{C}$ and that the two constructions establish an  equivalences between two categories of adjunctions.
\end{proof}

\begin{corollary}\label{indexed}
For an adjunction $L \dashv R : \mathcal{C}/X \pile { \rTo \\ \lTo} [G, \mathcal{C}]$ over $\mathcal{C}$ the following are equivalent:

(1) $L \dashv R$ is stably Frobenius,

(2) $L_{G^*X} \dashv R_{G^*X}$ satisfies Frobenius reciprocity; and,

(3) $L_{G^*Z} \dashv R_{G^*Z}$ satisfies Frobenius reciprocity for every object $Z$ of $\mathcal{C}$.
\end{corollary}
We do not use these characterizations below; they are include here because they can be applied to show that geometric morphisms between bounded toposes over a base topos $\mathbf{Set}$ can be characterised as $\mathbf{Loc}$-indexed adjunctions (in the sense of indexed category theory, e.g. B1 of \cite{Elephant}). It is hoped to make this the subject of a separate paper.
\begin{proof}
Clearly (1) implies (3) implies (2) because (3) and (2) are weaker conditions than (1). (2) implies (1) because if $L_{G^*X} \dashv R_{G^*X}$ satisfies Frobenius reciprocity then so does the adjunction $\mathcal{C}/X \pile{\rTo \\ \lTo} \mathcal{C}/X \times X  \pile{\rTo \\ \lTo} [G, \mathcal{C}]/G^*X$. This latter adjunction, as we have remarked in the proof of the theorem, is over $\mathcal{C}/X$ and so we may apply the `Finally' part of the Proposition \ref{X=1} to conclude that it is stably Frobenius.
\end{proof}

\section{Extending to Groupoids}
The above definitions and results can easily be generalised from groups to groupoids. If $\mathbb{G}=(G_1 \times_{G_0} G_1 \rTo^m G_1 \pile { \rTo^{d_0} \\ \rTo_{d_1}} G_0 )$ is an internal groupoid in a cartesian category $\mathcal{C}$ then $((G_1)_{d_0} , m)$ is itself a `special' object of $[ \mathbb{G} , \mathcal{C} ]$ in the sense that   $((G_1)_{d_0} , m) \times (A_g, *_A) \cong ((G_1)_{d_0} , m) \times \mathbb{G}^*A$ where $\mathbb{G}^*: \mathcal{C} \rightarrow [ \mathbb{G} , \mathcal{C} ]$ is the functor that send an object $X$ of $\mathcal{C}$ to the $\mathbb{G}$-object $ (\pi_1:G_0 \times X \rTo G_0, d_1 \times Id_X)$. The data for a principal $\mathbb{G}$-bundle additionally includes a map $g: P \rTo G_0$ that is invariant under the action. The proofs above go through essentially unchanged, so we content ourselves with stating the following theorem:
\begin{theorem}\label{groupoid}
$\mathcal{C}$ is a cartesian category and $\mathbb{G}$ is an internal groupoid with the property that the functor $\mathbb{G}^* : \mathcal{C} \rightarrow [ G, \mathcal{C}]$ has a left adjoint $\Sigma_{\mathbb{G}}$ such that ${\mathbb{G}}^* \dashv \Sigma_{\mathbb{G}}$ is stably Frobenius. $X$ is an object of $\mathcal{C}$. Then there is an equivalence between the category of principal $\mathbb{G}$-bundles over $X$ and the category of adjunctions $L \dashv R : \mathcal{C}/X \pile { \rTo \\ \lTo} [\mathbb{G}, \mathcal{C}]$ that are stably Frobenius and are over $\mathcal{C}$.
\end{theorem}

\section{Application to Geometric Morphisms}
We now apply our results to the case $\mathcal{C}=\mathbf{Loc}$, the category of locales and so $\mathbb{G}$ is a groupoid  internal to $\mathbf{Loc}$; i.e. a localic groupoid. See, for example, Part $C$ of \cite{Elephant} for relevant background material. Our aim is to explain how to apply the results above to show that geometric morphisms $f: Sh(X) \rTo B\mathbb{G}$ are the same thing as principal $\hat{\mathbb{G}}$-bundles over $X$, where $Sh(X)$ is the topos of sheaves for a locale $X$ and $B \mathbb{G}$ is the topos of $\mathbb{G}$-equivariant sheaves; that is, the full subcategory of $[ \mathbb{G}, \mathbf{Loc} ]$ consisting of $\mathbb{G}$-objects, $(A_g,*_A)$ such that $g:A \rTo G_0$ is a local homeomorphism. $\hat{\mathbb{G}}$ is the \'{e}tale completion of $\mathbb{G}$; see, e.g. C5.3.16 of \cite{Elephant} for a description of \'{e}tale completion. We will show that we cannot hope to apply the result for arbitrary localic groupoids $\mathbb{G}=(G_1 \pile { \rTo^{d_0} \\ \rTo_{d_1}} G_0 )$, but we can for the 2 important special cases of (i) an open and (ii) a proper localic groupoid; that is, $d_0$ (equivalently $d_1$) is (i) open and (ii) proper. To apply Theorem \ref{groupoid} we need to make two connections. Firstly we need to recall that geometric morphisms $f : \mathcal{F} \rTo \mathcal{E}$ between any two elementary toposes $\mathcal{F}$ and $\mathcal{E}$ can be represented as stably Frobenius adjunctions $\Sigma_f \dashv f^*$ between the corresponding categories of locales (that is, between $\mathbf{Loc}_{\mathcal{F}}$ and $\mathbf{Loc}_{\mathcal{E}}$). Secondly we need to recall what conditions are required to ensure that the equivalence $\mathbf{Loc}_{B\mathbb{G}}\simeq [ \hat{\mathbb{G}}, \mathbf{Loc}]$ holds (it is well known that $\mathbf{Loc}_{Sh(X)}\simeq \mathbf{Loc}/X$; e.g. Theorem C1.6.3 of \cite{Elephant}). The following two propositions address how to make these two connections in turn.

\begin{proposition}
For any two elementary toposes $\mathcal{F}$ and $\mathcal{E}$ there is a categorical equivalence between the category of geometric morphisms from $\mathcal{F}$ to $\mathcal{E}$ and the category of adjunctions $L \dashv R : \mathbf{Loc}_{\mathcal{F}} \pile{ \rTo \\ \lTo } \mathbf{Loc}_{\mathcal{E}}$ that are stably Frobenius and have $R$ preserving the Sierpi\'{n}ski locale.
\end{proposition}
\begin{proof}
This is essentially the main result of \cite{towgeom}. If $f: \mathcal{F} \rTo \mathcal{E}$ is a geometric morphism between elementary toposes then there is a `pullback' adjunction $\Sigma _{f}\dashv f^{\ast }$ between the category of locales in $\mathcal{F}$ and the category of locales in $\mathcal{E}$, with the right adjoint being given by pullback in the category of elementary toposes. \cite{towgeom} shows how C2.4.11 of \cite{Elephant} can be used to easily show that the adjunction $\Sigma _{f}\dashv f^{\ast }$ satisfies Frobenius reciprocity for any geometric morphism $f$ and, moreover, shows that any such adjunction, $L \dashv R$, arises in this way from a uniquely determined geometric morphism, provided $R$ preserves the Sierpi\'{n}ski locale and its internal distributive lattice structure. But for any locale $X$ over $\mathcal{E}$ there is a geometric morphism $f_X : Sh_{\mathcal{F}}(f^*X)  \rTo Sh_{\mathcal{E}}(X)$ obtained by pulling back along the localic geometric morphism $Sh(X) \rTo \mathcal{E}$. Lemma 3.2 of \cite{towgeom} confirms the easily observed fact that the pullback adjunction $\Sigma _{f_X}\dashv (f_X)^{\ast }$ is $(\Sigma _{f})_X\dashv (f^*)_X$ (under $\mathbf{Loc}_{Sh(X)} \simeq \mathbf{Loc}/X)$ and so $\Sigma _{f}\dashv f^{\ast }$ is stably Frobenius since $(\Sigma _f)_X \dashv (f^{\ast })_X$ satisfies Frobenius reciprocity for each $X$.
\end{proof}
For all localic groupoids $\mathbb{G}$, the functor $\mathbb{G}^*: \mathbf{Loc} \rightarrow [ \mathbb{G} , \mathbf{Loc} ]$ has a left adjoint since $\mathbf{Loc}$ has coequalizers. But the resulting adjunction does not necessarily satisfy Frobenius reciprocity. To see this, consider a regular epimorphism $f: X \rTo Y$ in the category of locales that is not stable under products (so, there exists a locale $Q$ such that $X \times Q  \rTo^{Id_Q \times f} Y \times Q $ is not a regular epimorphism - see p39, preamble to Lemma 4.4, of \cite{Plewe_Descent}, for a specific example of such $f$ and $Q$). Let $\mathbb{G}$ be the groupoid determined by the kernel pair of $f$. Then $\Sigma_{\mathbb{G}}(1)=Y$ and $\mathbb{G}^*Q$ is $(X \times Q, (X \times_Y X) \times Q \rTo^{\pi_2 \times Id_Q} X \times Q)$, and so $\Sigma_{\mathbb{G}}\mathbb{G}^*X$ is the coequalizer of the product of the kernel pair of $f$ and $Q$. By assumption this coequalizer is not $Y \times Q$ and so we cannot have $\Sigma_{\mathbb{G}} (1 \times \mathbb{G}^*(Q)) \cong \Sigma_{\mathbb{G}} (1) \times Q$ and $\Sigma_{\mathbb{G}} \dashv \mathbb{G}^*$ does not satisfy Frobenius reciprocity. So, ensuring that $\Sigma_{\mathbb{G}} \dashv \mathbb{G}^*$ is stably Frobenius must require some further assumptions of $\mathbb{G}$. The following proposition describes two cases of such further assumptions:
\begin{proposition}
If $\mathbb{G}$ is an open or proper localic groupoid then

(i) $\mathbf{Loc}_{B\mathbb{G}}\simeq [\hat{\mathbb{G}} , \mathbf{Loc}]$ over $\mathbf{Loc}$; and,

(ii) the adjunction $\Sigma_{\hat{\mathbb{G}}} \dashv {\hat{\mathbb{G}}}^*: [\hat{\mathbb{G}} , \mathbf{Loc}] \pile { \rTo \\ \lTo} \mathbf{Loc} $ is stably Frobenius.
\end{proposition}
\begin{proof}
(i) Theorem C5.1.5 of \cite{Elephant} shows that locales descend along geometric morphisms $f:\mathcal{F} \rTo \mathcal{E} $, whenever $f$ is an open surjection or a  proper surjection. For any localic groupoid $\mathbb{G}$ there is a surjective geometric morphism $d: Sh(G_0) \rTo B\mathbb{G}$ (whose inverse image is the forgetful functor), and it is easy to see that the definition of `locales descend along $d$' (see the preamble to Lemma 5.1.2 of \cite{Elephant}) is equivalent to the assertion that $\mathbf{Loc}_{B\mathbb{G}}\simeq [\hat{\mathbb{G}} , \mathbf{Loc}]$ because $\hat{\mathbb{G}}$ is by definition the localic groupoid determined by pulling back $d$ against itself (C5.3.16 of \cite{Elephant}).

Lemma C5.3.6 of \cite{Elephant} shows that for an open (or proper) localic groupoid $\mathbb{G}$ the geometric morphism $d$ is an open (or proper) surjection and so $\mathbf{Loc}_{B\mathbb{G}}\simeq [\hat{\mathbb{G}} , \mathbf{Loc}]$ as required.

The forgetful functor $[\hat{\mathbb{G}} , \mathbf{Loc}] \rTo   \mathbf{Loc}/G_0$ corresponds to $d^*:\mathbf{Loc}_{B\mathbb{G}} \rTo \mathbf{Loc}/G_0$ under this equivalence and since the forgetful functor is monadic, it reflects isomorphisms. Using $\gamma_{\mathbb{G}}$ for the geometric morphism $B\mathbb{G} \rTo \mathbf{Set}$, observe that $d^*\gamma^*_{\mathbb{G}} \cong G_0^*$ and so the equivalence $\mathbf{Loc}_{B\mathbb{G}}\simeq [\hat{\mathbb{G}} , \mathbf{Loc}]$ can be seen to be over $\mathbf{Loc}$ since $G_0$ is the locale of objects of $\hat{\mathbb{G}}$.

(ii) is clear from (i) because $\gamma_{\mathbb{G}}$ induces a stably Frobenius adjunction $\Sigma_{\gamma_{\mathbb{G}}} \dashv \gamma_{\mathbb{G}}^*: \mathbf{Loc}_{B\mathbb{G}} \pile { \rTo \\ \lTo} \mathbf{Loc} $ by the last Proposition and we have observed that $\gamma_{\mathbb{G}}^*$ maps to $\hat{\mathbb{G}}^*$ under $\mathbf{Loc}_{B\mathbb{G}}\simeq [\hat{\mathbb{G}} , \mathbf{Loc}]$

Alternatively, (ii) can be proved directly. If $\mathbb{G}$ is open (or proper) then so is its \'{e}tale completion (C5.3.16 of \cite{Elephant}). But asserting that the adjunction $\Sigma_{\mathbb{G}} \dashv \mathbb{G}^*$ is stably Frobenius can be seen to be equivalent to asserting that the coequalizer determined by $\Sigma_{\mathbb{G}}(A_{g},*_A)$ is pullback stable. This is well known to be the case if the groupoid is open or proper because the coequalizer determined by $\Sigma_{\mathbb{G}}(A_{g},*_A)$ must be open (e.g. Proposition C5.1.4 of \cite{Elephant}) and open (and proper) coequalizers are pullback stable. 

\end{proof}
\begin{remark}
It is worth noting that the direct proof of (ii) can be done axiomatically (using an axiomatic system similar to \cite{tow_weakt}). This shows that statements and results about open maps are formally dual to statements and results about proper maps. It also follows that we could apply our main result to $[\mathbb{G}, \mathbf{Loc}]$, without going to the \'{e}tale completion; but the cost is that $[\mathbb{G}, \mathbf{Loc}]$ will not necessarily be a category of locales for some topos. As future work it is may be worth examining the question of whether an axiomatic approach to locale theory is stable under the formation of the category of $\mathbb{G}$-objects, where $\mathbb{G}$ is not necessarily \'{e}tale complete. This could provide a category of `spaces' more granular than the category of bounded toposes and still capable of classifying principal bundles.
\end{remark}
We now state and prove our main application.
\begin{theorem}
Let $\mathbb{G}$ be a localic groupoid and $X$ a locale.

(i) If $\mathbb{G}$ is open, there is an equivalence between the category of geometric morphisms $Sh(X) \rTo B\mathbb{G}$ and the category of principal $\hat{\mathbb{G}}$-bundles over $X$. The principal bundle maps $f:P \rTo X$ that arise in this way are always open surjections.

(ii) If $\mathbb{G}$ is proper, there is an equivalence between the category of geometric morphisms $Sh(X) \rTo B\mathbb{G}$ and the category of principal $\hat{\mathbb{G}}$-bundles over $X$. The principal bundle maps $f:P \rTo X$ that arise in this way are always proper surjections.
\end{theorem}
Any Grothendieck topos is equivalent to $B\mathbb{G}$ for some open localic groupoid (C5.2.11 of \cite{Elephant}), so (i) provides a principal bundle description of the points (with localic domains at least) of arbitrary Grothendieck toposes. In fact one can always choose an \'{e}tale complete open localic groupoid to represent a Grothendieck topos (C5.3.16 \cite{Elephant}), and so for any Grothendieck topos $\mathcal{E}$ there is a localic groupoid $\mathbb{G}$ such that geometric morphisms $Sh(X) \rTo \mathcal{E}$ (over $\mathbf{Set}$) are the same things as principal $\mathbb{G}$-bundles over $X$.
(i) is known for \'{e}tale groupoids; that is, groupoids such that $d_0$ (equivalently $d_1$) is a local homeomorphism; \cite{IM_classtop}, \cite{IM_class_fourier}.
\begin{proof}
(i) and (ii) together: The proof is essentially a question of applying our main theorem (Theorem \ref{groupoid}), given the last two propositions. Notice for any adjunction $\mathbf{Loc}/X \pile{ \rTo \\ \lTo } \mathbf{Loc}_{B\mathbb{G}}$ that is over $\mathbf{Loc}$, the right adjoint must preserve the Sierpi\'{n}ski locale because both $\gamma_{\mathbb{G}}^*: \mathbf{Loc} \rTo \mathbf{Loc}_{B\mathbb{G}}$ and $X^*: \mathbf{Loc} \rTo \mathbf{Loc}/X$ preserve the Sierpi\'{n}ski locale.

For any principal bundle $(f: P \rTo X,*:G_1 \times_{G_0} P \rTo P)$ determined by either the equivalence of (i) or (ii), it should be clear that the morphism $f$ is an open (or proper) surjection. This is because it is determined by pullback of the open (proper) surjection $d: Sh(G_0) \rTo B \mathbb{G}$ and open (proper) surjections are pullback stable.
\end{proof}

\section{Further work}
There are two areas where more detailed further work should easily yield specific results:

1. Results of Moerdijk (\cite{IM_class_gp_II}) show how geometric morphisms can be described as certain locales with actions, and so are similar to our results. In that paper the actions are of a localic category, rather than a localic groupoid and so it is not immediately clear how to relate Moerdijk's results back to ours. However the key construction of \cite{IM_class_gp_II} also uses a tensor, similarly to our results, so there appears to be a close relationship.

2. In this paper we have only looked at geometric morphisms $Sh(X) \rTo B\mathbb{G}$ over $\bf{Set}$, rather than general geometric morphisms $\mathcal{F} \rTo B \mathbb{G}$. For $\mathcal{F}$ bounded over $\bf{Set}$ we can always find an open groupoid $\mathbb{H}$ so that such general geometric morphisms can be represented as stably Frobenius adjunctions between $[ \mathbb{H}, \mathbf{Loc} ]$ and $[ \mathbb{G}, \mathbf{Loc} ]$. It is expected that in a category whose objects are stably Frobenius adjunctions over some base cartesian category $\mathcal{C}$ (and whose morphisms are stably Frobenius adjunctions over $\mathcal{C}$), any object of the form $[ \mathbb{H}, \mathcal{C}]$ is a suitable coequalizer (perhaps of the simplicial diagram determined by $\mathbb{H}$). In this way it should be straightforward to extend the results from $Sh(X)$ to an arbitrary bounded topos $\mathcal{F}$, so providing a description of general geometric morphisms as a locale over 2 bases ($H_0$ and $G_0$) with two (interacting) groupoid actions such that one of the actions is principal.

The notion of a parallel theory of `proper' principal bundles, hinted at in the introduction, is more speculative, but has obvious appeal.


\begin{thebibliography}{test}

\bibitem[J02]{Elephant}  Johnstone, P.T. \emph{Sketches of an elephant: A topos theory compendium}. Vols 1, 2, Oxford Logic Guides \textbf{43},
\textbf{44}, Oxford Science Publications, 2002.

\bibitem[JT84]{JoyT}  Joyal, A. and Tierney, M. \emph{An Extension of the Galois Theory of Grothendieck}, Memoirs of the American Mathematical Society
\textbf{309}, 1984.

\bibitem[K89]{KockGen}  Kock, Anders. \emph{Fibre bundles in general categories}, Journal of Pure and Applied Algebra \textbf{56.3} (1989): 233-245.

\bibitem[I90]{IM_class_gp_II} Moerdijk, I. \emph{The classifying topos of a continuous groupoid. II.} Cahiers de topologie et geometrie differentielle categoriques, tome \textbf{31}, No. 2 (1990), 137-168

\bibitem[I91]{IM_class_fourier} Moerdijk, I. \emph{Classifying toposes and foliations. Annales de l'institut Fourier}, \textbf{41} no. 1 (1991), p. 189-209.

\bibitem[I96]{IM_classtop} Moerdijk, I. \emph{Classifying spaces and classifying topoi.} No. \textbf{1616}. Springer Verlag, (1995).

\bibitem[P97]{Plewe_Descent} Plewe, T. \emph{Localic triquotient maps are effective descent maps.} Mathematical Proceedings of the Cambridge Philosophical Society. Vol. \textbf{122}. No. 01. Cambridge University Press, (1997) 17-43.

\bibitem[P00]{Plewe_Quotient} Plewe, T. \emph{Quotient maps of locales.} Applied Categorical Structures \textbf{8}.1-2 (2000) 17-44.

\bibitem[T06]{towpara} Townsend, C.F. \emph{On the Parallel between the Suplattice and Preframe approaches to Locale Theory} Annals of Pure and Applied Logic, Volume \textbf{137}, Numbers 1-3 (2006) 391-412

\bibitem[T10]{tow_weakt} Townsend, C. F. \emph{An axiomatic account of weak triquotient assignments in locale theory.} Journal of Pure and Applied Algebra 214.\textbf{6} (2010): 729-739.

\bibitem[T10]{towgeom} Townsend, C.F. \emph{A representation theorem for geometric morphisms.} Applied Categorical Structures. \textbf{18} (2010) 573-583


\end{thebibliography}
\end{document}